\documentclass[10 pt]{amsart}

\usepackage[T1]{fontenc}
\usepackage[utf8]{inputenc}

\usepackage{xypic}
\usepackage{amsmath}
\usepackage{mathdesign}

\usepackage{enumerate}

\usepackage{amssymb}
\usepackage{dsfont}
\usepackage{tikz-cd}

\usepackage{extsizes}
\usepackage{capt-of}

\usepackage{lipsum}
\usepackage{footnote}

\usepackage[colorlinks]{hyperref}

\usepackage{fancyhdr}

\xyoption{all}

\newif\iffurther
\furthertrue

\newtheorem{thm}{Theorem}[section]

\newtheorem{cor}[thm]{Corollary}
\newtheorem{defn}[thm]{Definition}

\newtheorem{exmpl}[thm]{Example}
\newtheorem{lem}[thm]{Lemma}
\newtheorem{prop}[thm]{Proposition}

\newtheorem{rem}[thm]{Remark}

\def\[{\left[}
\def\]{\right]}

\long\def\forget#1\forgotten{{}}

\title[On the complexity of subshifts and infinite words]{On the complexity of subshifts and infinite words}

\author{Be'eri Greenfeld} \address{Department of Mathematics and Statistics, Hunter College at the City University of New York, 695 Park Avenue New York, NY 10065, USA} \email{beeri.greenfeld@hunter.cuny.edu}

\author{Carlos Gustavo Moreira} \address{SUSTech International Center for Mathematics, Shenzhen, Guangdong,
People’s Republic of China;
IMPA, Estrada Dona Castorina 110, 22460-320, Rio de Janeiro, Brazil} \email{gugu@impa.br}

\author{Efim Zelmanov} \address{SUSTech International Center for Mathematics, Shenzhen, Guangdong,
People’s Republic of China} \email{efim.zelmanov@gmail.com}

\begin{document}

\maketitle

\begin{abstract}
We characterize the complexity functions of subshifts up to asymptotic equivalence. The complexity function of every aperiodic subshift is non-decreasing, submultiplicative and grows at least linearly. We prove that conversely, every function satisfying these conditions is asymptotically equivalent to the complexity function of a recurrent subshift, equivalently, a recurrent infinite word. Our construction is explicit, algorithmic in nature and is philosophically based on constructing certain `Cantor sets of integers', whose `gaps' correspond to blocks of zeros.

We also prove that every non-decreasing submultiplicative function is asymptotically equivalent, up to a linear error term, to the complexity function of a minimal subshift.
\end{abstract}

\section{Introduction}

Let $\Sigma$ be a finite alphabet. A subshift $\mathcal{X}$ is a non-empty, closed, shift-invariant subspace of $\Sigma^{\mathbb{N}}$. The complexity function of $\mathcal{X}$ counts finite factors, that is, finite words which occur as subwords of $\mathcal{X}$:
$$
p_\mathcal{X}(n) = \# \{ \text{Factors of}\ \mathcal{X}\ \text{of length}\ n\} = \# \{u\in \Sigma^n\ |\ \exists w\in \mathcal{X}\ \text{s.t.} \ w = \cdots u \cdots \}. 
$$

An important special case is when $\mathcal{X}$ is a transitive subshift, that is, the closure of the shift-orbit of a single infinite word: $$ w=w_0 w_1 w_2 \cdots \in \Sigma^{\mathbb{N}}. $$ 
In this case, the complexity function counts the finite factors of $w$, and is denoted:
$$
p_w(n) = \# \{ \text{Factors of}\ w\ \text{of length}\ n\} =  \# \{ w_i \cdots w_{i+n-1}\ |\ i=0,1,2,\dots \}.
$$
It has been a long-standing open question to characterize the complexity functions of subshifts and specifically of infinite words \cite{Cas, Fer99, Fer96} (Cassaigne \cite{Cas} describes this problem as ``challenging and probably very difficult''), or even only the possible orders of magnitudes of complexity functions of infinite words \cite{MM1}.

The complexity function of every subshift is:
\begin{enumerate}[(i)]
    \item Non-decreasing: $p_{\mathcal{X}}(n)\leq p_{\mathcal{X}}(n+1)$
    \item Submultiplicative: $p_\mathcal{X}(n+m)\leq p_{\mathcal{X}}(n)p_{\mathcal{X}}(m)$
    \item Either bounded ($p_{\mathcal{X}}(n)\leq C$) or at least linear ($p_{\mathcal{X}}(n)\geq n+1$), by the Morse-Hedlund theorem \cite{MH38}.
\end{enumerate}

Many partial results have been proven in the converse direction, namely, realizing functions satisfying the above conditions as complexity functions of infinite words, see \cite{Alo, Cas2, Cas, Cas3, Fer99, Fer96} and references therein. Notice that the above conditions are not sufficient for a function to be realizable as the complexity function of an infinite word (see \cite{Cas3}).

Our main result shows that the three aforementioned conditions are in fact sufficient for a function to be asymptotically equivalent as the complexity function of a subshift. 
Let $f,g\colon \mathbb{N}\rightarrow \mathbb{N}$ be non-decreasing functions. We write $f\preceq g$ if $f(n)\leq Cg(Dn)$ for some constants $C,D>0$, and $f\sim g$ (asymptotically equivalent) if both $f\preceq g$ and $g\preceq f$. 
The notion of asymptotic equivalence is fundamental in large scale geometry, geometric group theory and combinatorial algebra.  In particular, growth functions of groups, semigroups, associative algebras, and Lie algebras are independent of the choice of generating set only up to asymptotic equivalence. For related results, see \cite{BZ, Gri2, Gromov, Gromov2, SB, Tro}.
In the context of complexity functions, asymptotic equivalence identifies `weighted' complexity functions within the same subshift, allowing for assigning arbitrary positive real weights to each alphabet letter.

\begin{thm} \label{thm:main}
Let $f\colon \mathbb{N}\rightarrow \mathbb{N}$ be a non-decreasing function, satisfying $f(n)\geq n+1$ and $f(2n)\leq f(n)^2$ for every $n\in \mathbb{N}$. Then there exists an infinite word $w$ such that $f\sim p_w$. Moreover, $w$ can be taken to be a recurrent word.
\end{thm}

Notice that non-trivial outcomes of Theorem \ref{thm:main} is that the complexity function of every subshift is asymptotically equivalent to the complexity function of a single infinite word, and furthermore of a recurrent word -- that is, an infinite word in which every finite factor appears infinitely many times -- which is of high importance in symbolic dynamics.

    The proof of this theorem depends on an explicit algorithmic construction, which is philosophically based on constructing certain `Cantor sets of integers', whose `gaps' correspond to blocks of zeros; the structure of these `Cantor sets' may oscillate, some times resembling the usual middle-third Cantor set, and some times resembling a positive measure Cantor set, depending on the desired growth of the complexity function. Constructions of infinite words based on `positive measure Cantor sets of integers' were used by Mauduit and by the second named author in \cite{MM1, MM2} to study sets of infinite words with complexity functions bounded by a given function.

A wider framework for complexity functions is given by formal languages. The growth function of a formal language is counts words in it of length at most $n$, for each $n\in \mathbb{N}$. The set of finite factors of any subshift is a formal language which is hereditary (that is, closed under taking subwords). However, there exist hereditary languages whose growth functions are not asymptotically equivalent to those of any subshift. Growth functions of hereditary languages coincide with the growth functions of associative algebras and semigroups. Interesting realization theorems of different functions as growth functions of hereditary languages (equivalently, algebras or semigroups) were found in \cite{BBL, SB, Tro}; a complete characterization of growth functions of hereditary languages, to which Theorem \ref{thm:main} is analogous, was finally done by Bell and by the third named author in \cite{BZ}. Interesting examples of hereditary languages with highly oscillating (in the strongest possible sense) growth functions were constructed by Belov, Borisenko and Latyshev in \cite{BBL}; we refer the reader to their book for more on the connection between infinite words and associative algebras.

A subshift is minimal if it admits no proper subshifts. Every infinite word in any minimal subshift has a dense shift-orbit inside it, and in combinatorial terms, is uniformly recurrent -- that is, every finite factor of it appears within every other sufficiently large factor of it; this is a stronger property than recurrence. 
We prove an approximate realization result for complexity function of minimal subshifts:

\begin{prop} \label{prop:minimal}
Let $f\colon \mathbb{N}\rightarrow \mathbb{N}$ be a non-decreasing, unbounded, submultiplicative function. Then  there exists a minimal subshift $\mathcal{X}$ such that $f(n) \preceq p_{\mathcal{X}}(n) \preceq n f(n)$.
\end{prop}

Finally, we derive some applications to the growth of algebras.

\section*{Acknowledgement}
We are grateful to the referee for a careful and thorough reading of our manuscript, and for many comments, corrections and suggestions which significantly improved the paper.

\section{Complexity functions of subshifts, words and algebras}

Let $\Sigma$ be a finite alphabet. A hereditary language over $\Sigma$ is a non-empty collection of finite words $\mathcal{L}\subseteq \Sigma^*=\bigcup_{n=0}^{\infty} \Sigma^n$, closed under taking subwords. A hereditary language $\mathcal{L}$ is (right) extendable if for every $\alpha\in \mathcal{L}$ there exists some non-empty word $\alpha'$ such that $\alpha \alpha' \in \mathcal{L}$. 
A subshift over $\Sigma$ is a non-empty, closed subset of $\mathcal{X}\subseteq \Sigma^{\mathbb{N}}$ with respect to the product topology, which is invariant under the shift operator:
\begin{eqnarray*} T \colon \Sigma^{\mathbb{N}} & \longrightarrow & \Sigma^{\mathbb{N}} \\  w_0w_1w_2\cdots 
 & \mapsto & w_1w_2w_3\cdots \end{eqnarray*}
The set of finite factors of every subshift $\mathcal{X}$, denoted $\mathcal{L}(\mathcal{X})$ is an extendable hereditary language; conversely, every extendable hereditary language is the language of finite factors of a uniquely determined subshift.

The notion of complexity readily extends to formal languages which are not necessarily extendable. If $L\subseteq \Sigma^*$ is a (hereditary) language, let $p_{L}(n) = \#\left(L \cap \Sigma^n\right)$. Clearly, any function $f\colon \mathbb{N}\rightarrow \mathbb{N}$ occurs as the complexity functions of some formal language; the case of hereditary languages is much more restricted and interesting. Since complexity functions of non-extendable languages are not necessarily non-decreasing, it is useful to introduce the notions of growth functions of hereditary languages as the global counting function of the complexity function: $P_L(n) = \sum_{i=0}^{n} p_L(i) = \#\left(L\cap \Sigma^{\leq n}\right)$. Similarly, one can define the growth function of a subshift by $P_\mathcal{X} = P_{\mathcal{L}(\mathcal{X})}$. The class of growth functions of hereditary languages coincides with the class of growth functions of associative algebras and of semigroups (for more on this, see \cite{Ber,KL}). The discrete derivative of a function $f\colon \mathbb{N}\rightarrow \mathbb{N}$ is given by $f'(n)=f(n)-f(n-1)$ for $n>1$, and $f'(0)=f(0)$. Thus, $p_L=P'_L$.

A hereditary language $L$ is recurrent (or irreducible) if for every $\alpha_1,\alpha_2\in L$ there exists some $\beta\in L$ such that $\alpha_1 \beta \alpha_2\in L$. Recurrent hereditary languages are extendable and correspond to recurrent subshifts and to recurrent infinite words. In fact, our strategy to construct the prescribed infinite words in Theorem \ref{thm:main} is to construct a hereditary language of desired complexity and then show that it is recurrent.  Algebraically, recurrent hereditary languages define precisely those monomial algebras which are prime.

For the reader's convenience, we repeat the basic necessary properties satisfied by any complexity function of a subshift.

\begin{prop}
Let $\mathcal{X}$ be a subshift over a finite alphabet $\Sigma$. Then: (i) $p_\mathcal{X}$ is non-decreasing; (ii) $p_\mathcal{X}(n+m)\leq p_\mathcal{X}(n)p_\mathcal{X}(m)$ for every $n,m\in \mathbb{N}$, and (iii) either $p_\mathcal{X}$ is bounded or $p_\mathcal{X}(n)\geq n+1$ for every $n$.
\end{prop}

\begin{proof}
Let $\mathcal{L}$ be the hereditary language of finite factors of $\mathcal{X}$. For each $w\in \mathcal{L}\cap \Sigma^n$, fix some (any) word $w'\in \mathcal{L}\cap \Sigma^{n+1}$ whose length-$n$ prefix is $w$; this is possible since $w$ factors an infinite word from $\mathcal{X}$. The function $w\mapsto w'$ is an injective function $\mathcal{L}\cap \Sigma^n\rightarrow \mathcal{L}\cap \Sigma^{n+1}$, proving (i). Furthermore, every length-$n+m$ factor of $\mathcal{X}$ is uniquely determined by its length-$n$ prefix and its length-$m$ suffix, both of which are still factors of $\mathcal{X}$, proving (ii). Finally, (iii) is the Morse-Hedlund theorem \cite{MH}.
\end{proof}

We freely use the following standard asymptotic properties of functions throughout.

\begin{lem} \label{lem:enough n_k}
Let $f,g\colon \mathbb{N}\rightarrow \mathbb{N}$ be non-decreasing functions and let $(n_k)_{k=1}^{\infty}$ be a sequence of positive integers satisfying $n_k < n_{k+1} \leq C n_k$ for some constant $C > 1$. If $f(n_k)\leq Dg(Dn_k)$ for some constant $D>0$ then $f\preceq g$.
\end{lem}

\begin{proof}
Let $n \in \mathbb{N}$ be arbitrary. Fix $k$ such that $n_k \leq n \leq n_{k+1}$. Then $f(n)\leq f(n_{k+1}) \leq Dg(Dn_{k+1}) \leq Dg(CD n_k) \leq Cg(CD n)$.
\end{proof}

\begin{lem}
Let $f,g\colon \mathbb{N}\rightarrow \mathbb{N}$ be non-decreasing functions with non-decreasing discrete derivatives. If $f'(n)\leq Cg'(Cn)$ for some $C>0$ then $f(n)\leq g(Dn)$ for some $D>0$.
\end{lem}

\begin{proof}
$f(n) = \sum_{i=0}^n f'(i) \leq C\sum_{i=0}^n g'(Ci) \leq \sum_{i=0}^{2C^2 n} g'(i) = g(2C^2n)$.
\end{proof}

Theorem \ref{thm:main} can be reformulated to characterize the growth functions of subshifts and infinite words:

\begin{thm}[{A reformulation of Theorem \ref{thm:main}}]
Let $F\colon \mathbb{N}\rightarrow \mathbb{N}$ be a non-decreasing function, such that $F'$ is non-decreasing, $F'(n)\geq n+1$ and $F'(2n)\leq F'(n)^2$ for every $n\in \mathbb{N}$. Then there exists a recurrent infinite word $w$ such that $F\sim P_w$.
\end{thm}
By the above lemma, in this reformulation the asymptotic equivalence notion does not require the outer constant.
This version of Theorem \ref{thm:main} is analogous to the main result in \cite{BZ}, which characterizes growth functions of hereditary languages. The assumptions in Theorem \ref{thm:main} are stronger than those in \cite{BZ}, which is necessary to realize given functions as complexity or growth functions of infinite words and not just hereditary languages. Indeed, there exist hereditary languages whose growth functions are not asymptotically equivalent to the growth function of any extendable language (e.g. \cite[Theorem 1.4]{BGSel}). Alternatively, this follows from the following.

\begin{exmpl}
Let $F\colon \mathbb{N}\rightarrow \mathbb{N}$ be given by $F(0)=1$ and, for every integer $n\ge 1$,
$$ F'(n) = F(n) - F(n-1) =  \left\{
\begin{array}{ll} (2k+2)! &\ \text{if}\ (2k-1)!\le n<(2k)! \\
      (2k+1)!^2 &\ \text{if}\ (2k)!\le n<(2k+1)! \\
\end{array} 
\right. $$
We have $F'(n)\geq n+1$ and $F'(m)\leq F'(n)^2$ for every $n\ge 1$ and $m\in\{n, n+1,\dots,2n\}$, and therefore by \cite[Corollary 4.5]{BZ}, $F$ is asymptotically equivalent to the growth function of some hereditary language. 

However, $F'$ is not a non-decreasing function, and indeed $F$ is not asymptotically equivalent to the growth function of any extendable language. In order to see this, notice that if there is a constant $C>1$ such that $F(n/C)<G(n)<F(Cn)$ for some growth function $G$ of an extendable language, we should have $G'$ non-decreasing, but for $k$ large, we would have the following estimates:
\begin{eqnarray*}
  G\left(\frac1{C}(2k+2)!\right)-G(C(2k+1)!) & \le &  F((2k+2)!)-F((2k+1)!) \\ & = & (2k+1)(2k+1)!(2k+4)!\,  
\end{eqnarray*}
and:
\begin{eqnarray*}
G(C(2k+1)!)-G\left(\frac1{C}(2k)!\right) & \ge & F((2k+1)!)-F((2k)!) \\ & = & (2k)(2k)!(2k+1)!^2,    
\end{eqnarray*}
from which we conclude that the average of the discrete derivative $G'$ of $G$ in the interval $[C(2k+1)!,\frac1{C}(2k+2)!)]$ is $\le (1+o(1))C(2k+4)!$, while the the average of $G'$ in the interval $[\frac1{C}(2k)!,C(2k+1)!)]$ is $\ge \frac{1+o(1)}{C}(2k+1)!^2$, which is much greater than $(1+o(1))C(2k+4)!$, a contradiction, since $G'$ is non-decreasing.
\end{exmpl}

\section{Construction} \label{sec:construction}

Let $f\colon \mathbb{N}\rightarrow \mathbb{N}$ be as in the assumption of Theorem \ref{thm:main}. Since $f(n)\geq n+1$ then $f(Cn)\geq Cn+1$ and so, up to equivalence, we may assume that $f(n)\geq 8n$, replacing $f(n)$ by $g(n):=f(8n)$. Notice that $g(2n)=f(16n)\leq f(8n)^2=g(n)^2$, so the submultiplicativity condition remains and there is no loss of generality in assuming $f(n)\geq 8n$. (In fact, we may even assume that $f$ is increasing, replacing it by $g(n)=f(n)+n$.) We may further assume that $f(0)=1$.
    Let $b=f(1)$, which by the above, can be assumed to be $\geq 8$. Let $\Sigma=\{0,1,\dots,\lfloor \frac{b}{2} \rfloor\}$ be our alphabet. Define two sequences of positive integers, $(n_k)_{k=1}^{\infty}$ and $(s_k)_{k=1}^{\infty}$. For each $k\geq 1$ we will define $X_k\subseteq \Sigma^{n_k} \setminus \left( 0\Sigma^* \cup \Sigma^*0 \right)$, namely, a set of length-$n_k$ words over the alphabet $\Sigma$ not starting or ending with $0$, of cardinality $|X_k|=s_k$. We think of each $X_k$ as an ordered set. Furthermore, for every $k\geq 1$ we will have $2n_k < n_{k+1} \leq 8n_k$. Eventually, we will take $L$ to be the set of all finite factors of $\bigcup_{k=1}^{\infty} X_{k}$, namely, its hereditary closure. Then, it will be shown that $L$ coincides with the set of finite factors of a certain infinite recurrent word, whose shift orbit closure is the desired subshift $\mathcal{X}$.
To start, we let $n_1=1$, $s_1=\lfloor \frac{b}{2} \rfloor$ and $X_1=\{1,2,\dots,\lfloor \frac{b}{2} \rfloor\}$. We proceed by induction, assuming $n_1,\dots,n_k,s_1,\dots,s_k,X_1,\dots,X_k$ have been defined.

\begin{figure}
\begin{tikzcd}
\text{Case I: One step} \arrow[loop, distance=2em, in=125, out=55] \arrow[ddd, shift right=2] \arrow[dddddd, bend right, shift right=5]                  &                                                                                                                           &                                                                                                                            \\
                                                                                                                                                         &                                                                                                                           &                                                                                                                            \\
                                                                                                                                                         &                                                                                                                           & \text{Sub-Case (1)} \arrow["\substack{\text{({Possibly}} \\ \ \infty\ \text{steps})}"', loop, distance=2em, in=125, out=55] \arrow[dd] \\
\text{Case II} \arrow[ddd, shift right=2] \arrow[uuu, shift right] \arrow[r, shift right]                                                                & \text{Preparation step} \arrow[ru] \arrow[rd] \arrow[luuu, shift right=2] \arrow[l, shift right] \arrow[lddd, shift left] &                                                                                                                            \\
                                                                                                                                                         &                                                                                                                           & \substack{\text{Sub-Case (2)} \\ \text{one step or route to III}} \arrow[lldd] \arrow[lluuuu, bend right=60]               \\
                                                                                                                                                         &                                                                                                                           &                                                                                                                            \\
\text{Case III} \arrow["(<\infty\ \text{steps})"', loop, distance=2em, in=305, out=235] \arrow[uuu, shift right] \arrow[uuuuuu, bend left, shift left=2] &                                                                                                                           &                                                                                                                           
\end{tikzcd}

\caption{A flow diagram of the different cases and routines used in the proof of Theorem \ref{thm:main}.}
\end{figure}
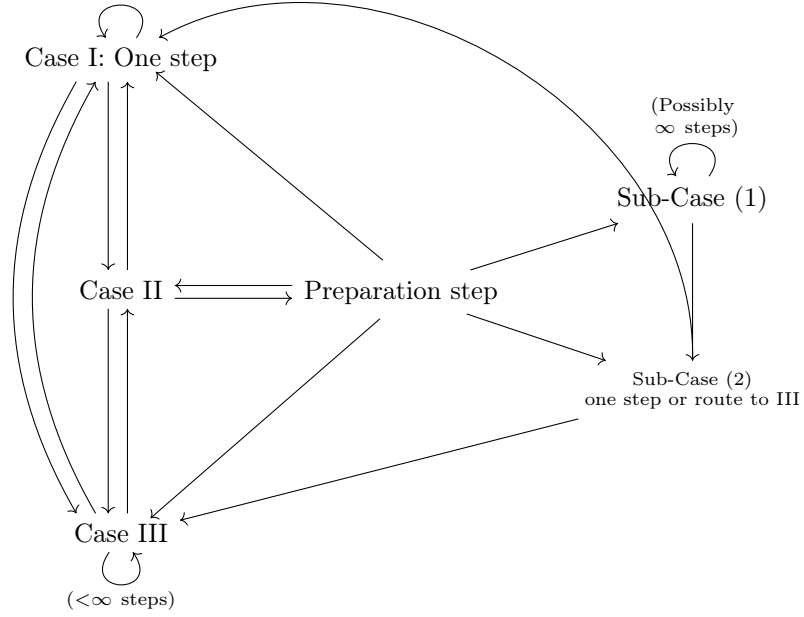

\begin{defn}
We say that $k$ is \textit{balanced} if: $$\frac{1}{3}f\left(\lfloor \frac{n_k}{3} \rfloor\right) \leq 2n_k s_k \leq f(n_k).$$
\end{defn}

By the end of each one of the cases, we end up with a balanced index. We thus assume that we are in  a balanced index in the beginning of each case. Notice that by the definition of $n_1,s_1$, it holds that $k=1$ is balanced.

\bigskip

\noindent \textit{Case I: Tame growth of $f$.} If:
\begin{align}
\sqrt{\frac{f(3n_k)}{6n_k}} \leq s_k \leq \frac{f(3n_k)}{6n_k}
\end{align}
then we take $n_{k+1}=3n_k$ and $s_{k+1}=\lfloor \frac{f(3n_k)}{6n_k} \rfloor = \lfloor \frac{f(n_{k+1})}{2n_{k+1}} \rfloor$ (notice that $k+1$ is balanced) so $s_{k+1}\geq s_k$ and pick, for each $\alpha\in X_k$, an arbitrary subset $Y(\alpha)\subseteq X_k$ of cardinality $|Y(\alpha)|=:r_\alpha \in \{\lfloor \frac{s_{k+1}}{s_k} \rfloor,\lceil \frac{s_{k+1}}{s_k} \rceil\}$, chosen such that $\sum_{\alpha\in X_k} r_\alpha = s_{k+1}$. This is indeed possible, since $\lfloor \frac{s_{k+1}}{s_k} \rfloor s_k \leq s_{k+1} \leq \lceil \frac{s_{k+1}}{s_k} \rceil s_k$, and: 
$$
r_\alpha \leq \lceil \frac{s_{k+1}}{s_k} \rceil \leq \lceil \frac{f(3n_k)}{6n_ks_k} \rceil \leq \lceil \sqrt{\frac{f(3n_k)}{6n_k}} \rceil \leq s_k = |X_k|.
$$
Now let:
$$ 
X_{k+1} = \{ \underbrace{\alpha}_{n_k} \underbrace{0^{n_k}}_{n_k} \underbrace{\beta}_{n_k}\ |\ \alpha\in X_k,\ \beta\in Y(\alpha) \} \subseteq \Sigma^{3n_k} = \Sigma^{n_{k+1}}
$$
and notice that $|X_{k+1}| = \sum_{\alpha\in X_k} |Y(\alpha)| =  \sum_{\alpha\in X_k} r_\alpha = s_{k+1}$.

We further take $\alpha_{i+1} \in Y(\alpha_i)$ for each $1\leq i\leq s_k$, where $\alpha_{s_k+1}$ is taken to be $\alpha_1$ (i.e. we work modulo $s_k$).

\bigskip

For Cases II,III, whenever we reach a balanced index, we terminate the routine.

\noindent \textit{Case II: Rapid growth of $f$.} If:
\begin{align}
s_k < \sqrt{\frac{f(3n_k)}{6n_k}}
\end{align}
then, equivalently, $f(3n_k) > 6n_ks_k^2$. 

\smallskip

\noindent \textit{Preparation step.}
We take $n_{k+1}=8n_k$ and: 
$$ s_{k+1} = \min \Big\{ \lfloor \frac{f(8n_k)}{16n_k} \rfloor, \max \{ s_k , f\left(\lfloor \frac{n_k}{3}\rfloor \right) \} \Big\}. $$
We now explain how to construct $X_{k+1}$. 
Notice that, since we assume that $k$ is balanced (as an output of one of the three cases) then: $$ s_{k+1} \leq \max\{s_k , f(\lfloor \frac{n_k}{3} \rfloor)\} \leq 6 n_k s_k. $$ 
Furthermore, notice that by the case assumption and by the assumptions on $f$, we have: $$ s_k < \sqrt{\frac{f(3n_k)}{6n_k}} \leq \sqrt{\frac{f(8n_k)}{6n_k}\cdot \frac{f(8n_k)}{48n_k}} \leq \frac{f(8n_k)}{16n_k} $$
so $s_k \leq s_{k+1}$.

We can thus take $X_{k+1}$ to be an arbitrary size-$s_{k+1}$ subset of $\Sigma^{n_{k+1}}=\Sigma^{8n_k}$ consisting of words of the form:
$$ \underbrace{\alpha}_{n_k}\underbrace{0^{n_k}}_{n_k} \underbrace{0^j\beta 0^{3n_k-j}}_{4n_k} \underbrace{0^{n_k}}_{n_k} \underbrace{\gamma}_{n_k} $$
where, if $\alpha=\alpha_p$ then $\gamma=\alpha_{p+1}$, taken cyclically over $\alpha_1,\dots,\alpha_{s_k}$.
Notice that, letting $j\in [0,3n_k]$, and $\alpha,\beta\in X_k$, we can arrange up to $3n_ks_k^2 \geq 6n_k s_k$ elements in $X_{k+1}$ with prefix $\alpha$ for each $\alpha\in X_k$, and altogether we can indeed take $X_{k+1}$ to be of size $s_{k+1}$ as required. The resulting words are all distinct as $\beta$ does not start or end with $0$.

In the case that $s_{k+1} = \lfloor f(8n_k)/16n_k  \rfloor = \lfloor \frac{f(n_{k+1})}{2n_{k+1}} \rfloor$, it follows that:
$$ f(n_{k+1}) \geq 2n_{k+1} \lfloor \frac{f(n_{k+1})}{2n_{k+1}} \rfloor = 2n_{k+1} s_{k+1} \geq f(n_{k+1}) - 2n_{k+1} \geq \frac{2}{3}f(n_{k+1}) \geq \frac{1}{3}f(\lfloor \frac{n_{k+1}}{3} \rfloor) $$ 
so $k+1$ is balanced and the routine stops. We thus concentrate, for the rest of Case II, on the case that $s_{k+1} = \max \{ s_k , f(\lfloor n_k/3 \rfloor ) \} \leq \lfloor f(8n_k)/16n_k \rfloor$.

For each $i\geq 1$ we define\footnote{A comment on the construction: the reason we divide by $(k+i)^2$ is that the series of reciprocals of square integers is convergent, and the squares grow slower than an exponential function. This will become more clear in the proof of Lemma \ref{lem:complexity bounds}.}:

\begin{align} \label{g and n}
g_{k+i} := & n_{k+i} - 2n_{k+i-1}  \\
\tilde{g}_{k+i+1} := & \max\{g_{k+i},\lceil \frac{n_{k+i}}{(k+i)^2}  \rceil\}, \nonumber \\  \tilde{n}_{k+i+1} := & 2n_{k+i}+\tilde{g}_{k+i+1}. \nonumber
\end{align}
(where $g_{k+1}=n_{k+1}-2n_k=8n_k-2n_k=6n_k$; notice it is smaller than $n_{k+1}=8n_k$. In addition, $\tilde{g}_{k+i+1}\leq n_{k+i+1}$ as will follow from the way we pick $n_{k+i+1}$; see below.)

\smallskip

\noindent \textit{Sub-Case (1):}
If $\frac{f(\tilde{n}_{k+i+1})}{2\tilde{n}_{k+i+1}} >  s_{k+i}^2$, we define $s_{k+i+1} := s_{k+i}^2$ and $n_{k+i+1} := \tilde{n}_{k+i+1}$ and notice that $g_{k+i+1} = \tilde{g}_{k+i+1}$. We now define $X_{k+i+1}$ as follows:
$$ X_{k+i+1} = \Big\{ \underbrace{\alpha}_{n_{k+i}} \underbrace{0^{g_{k+i+1}}}_{g_{k+i+1}} \underbrace{\beta}_{n_{k+i}}\ |\ \alpha,\beta\in X_{k+i} \Big\} \subseteq \Sigma^{2n_{k+i}+g_{k+i+1}} = \Sigma^{n_{k+i+1}} $$
of cardinality $s_{k+i}^2 = s_{k+i+1}$. We continue the routine, until (if ever) we get to Sub-Case (2).

\smallskip

\noindent \textit{Sub-Case (2):}
If $\frac{f(\tilde{n}_{k+i+1})}{2\tilde{n}_{k+i+1}} \leq  s_{k+i}^2$, we proceed as follows:

\begin{itemize}     
\item If $\sqrt{\frac{f(3n_{k+i})}{6n_{k+i}}} \leq s_{k+i} \leq \frac{f(3n_{k+i})}{6n_{k+i}}$, we set $n_{k+i+1} = 3n_{k+i}$ and $s_{k+i+1} = \lfloor \frac{f(3n_{k+i})}{6n_{k+i}} \rfloor$, and define $X_{k+i+1}$ as in Case I. Notice that in this case, we end up with a balanced index, and so this routine has ended, and we can continue to any of Cases I,II,III.

\item If $s_{k+i}^2 < \frac{f(3n_{k+i})}{6n_{k+i}}$, take:
$$ M = \min\Big\{ m\in [\tilde{n}_{k+i+1},3n_{k+i}]\ \Big|\ \frac{f(m)}{2m} > s_{k+i}^2\Big\}.$$

Notice that by Sub-Case (2) assumption, $M\neq \tilde{n}_{k+i+1} =  2n_{k+i}+\tilde{g}_{k+i+1}$. 
Notice that always $g_{k+i}\leq n_{k+i}$. Notice that $m=3n_{k+i}$ satisfies the above condition, so $M$ is well-defined.

Now define $n_{k+i+1} := M$ and $s_{k+i+1} = s_{k+i}^2$, so:  $$ \frac{f(n_{k+i+1}-1)}{2(n_{k+i+1}-1)} \leq s_{k+i+1} < \frac{f(n_{k+i+1})}{2 n_{k+i+1}} $$ 
(where the left inequality follows by the definition of $M$ and since $M\neq \tilde{n}_{k+i+1} = 2n_{k+i}+\tilde{g}_{k+i+1}$.)
We take:
$$ X_{k+i+1} = \Big\{ \underbrace{\alpha}_{n_{k+i}} \underbrace{0^{M-2n_{k+i}}}_{M-2n_{k+i}} \underbrace{\beta}_{n_{k+i}}\ |\ \alpha,\beta\in X_{k+i} \Big\} \subseteq \Sigma^M = \Sigma^{n_{k+i+1}} $$
of cardinality $s_{k+i}^2 = s_{k+i+1}$.

Also notice that $k+i+1$ is balanced, since:
\begin{eqnarray*}
f(\lfloor \frac{n_{k+i+1}}{3} \rfloor) \leq  f(n_{k+i+1} - 1) & \leq & 2n_{k+i+1}s_{k+i+1} \\ & = & 2Ms_{k+i}^2 \leq f(M) = f(n_{k+i+1}).    
\end{eqnarray*}
We thus end the routine and the case, and we can move on to any of Cases I,II,III.

\item If $s_{k+i} > \frac{f(3n_{k+i})}{6n_{k+i}}$, we claim that $k+i$ is balanced and that we can move to Case III and define $n_{k+i+1},s_{k+i+1}$ and define $X_{k+i+1}$ accordingly.

To see why $k+i$ is balanced, first notice that by the item's assumption, $$ 2 s_{k+i} n_{k+i} > \frac{1}{3} f(3n_{k+i}) \geq \frac{1}{3}f(\lfloor \frac{n_{k+i}}{3} \rfloor ). $$
Regarding the other direction, assume first that $i=1$. Then by definition $s_{k+1}\leq \frac{f(8n_k)}{16n_k} = \frac{f(n_{k+1})}{2n_{k+1}}$. If $i>1$, then we may assume that $n_{k+i},s_{k+i}$ are outputs of Sub-Case (1), as the above items of Sub-Case (2) end with the next index being balanced. By the way we define the next $n_*,s_*$ in Sub-Case (1), it follows that $s_{k+i} \leq \frac{f(n_{k+i})}{2n_{k+i}}$.
\end{itemize}

\bigskip

\noindent \textit{Case III:  Slow growth of $f$.} If:
\begin{align}
\frac{f(3n_k)}{6n_k} < s_k
\end{align}
we proceed as follows. For each $i\geq 0$ (until a termination condition, to be specified in the sequel, is satisfied) we take: $$ n_{k+i+1} = 2n_{k+i}+n_k = (2^{i+2}-1)n_k,\ \ s_{k+i+1}=\lceil\frac{s_{k+i}}{2}\rceil=\lceil \frac{s_k}{2^{i+1}}\rceil. $$
We now construct $X_{k+i}\subseteq \Sigma^{n_{k+i}}$ for each $i\geq 1$ (until the process is terminated). Write $s_k = 2^i m_k +j_k$ for some $m_k\in \mathbb{N}$ and $0<j_k\leq 2^i$ (so that $ s_{k+i} =  \lceil \frac{s_k}{2^i} \rceil = m_k+1$) and enumerate $X_k=\{\alpha_1,\dots,\alpha_{s_k}\}$. We now let: $$ \beta_{j+1} = \alpha_{2^i j +1}0^{n_k}\alpha_{2^i j +2} 0^{n_k} \cdots 0^{n_k} \alpha_{2^i j + 2^i} $$
for each $0\leq j\leq m_k - 1$, and: $$ \beta_{m_k+1} = \alpha_{2^i m_k + 1} 0^{n_k} \cdots 0^{n_k} \alpha_{s_k} 0^{n_k} \alpha_1 0 ^{n_k} \cdots 0^{n_k} \alpha_{2^i - j_k} $$
(if $j_k=2^i$ then $\beta_{m_k+1}$ ends with $\alpha_{s_k}$.) 
Observe that the length of each $\beta_j$ is $2^in_k+(2^i-1)n_k=(2^{i+1}-1)n_k=n_{k+i}$. We finally take $X_{k+i} = \{\beta_1,\dots,\beta_{m_k+1}\}\subseteq \Sigma^{n_{k+i}}$.
We proceed in this manner until the first \emph{positive} $i$ for which $s_{k+i} \le \frac{f(n_{k+i})}{2n_{k+i}}$. We then exit the sub-routine of Case III with $n_{k+i},s_{k+i}$ defined. Indeed, this is a well-defined termination condition as the right hand side is bounded from below by a constant ($\ge 4$) and the left hand side (namely $s_{k+i}$) decreases exponentially with $i$. 

Let us now explain why $s_{k+i}\geq 2$, and afterwards, why $k+i$ is indeed balanced.
Notice that $s_k>\frac{f(3n_k)}{6n_k}\geq 4$ since $f(x)\geq 8x$, so if $i=1$ then $s_{k+1} = \lceil \frac{s_k}{2} \rceil \geq 2$. If $i>1$, this means that $s_{k+i-1} > \frac{f(n_{k+i-1})}{2n_{k+i-1}}\geq 4$ (again, $f(x)\geq 8x$) and by the same reasoning, $s_{k+i} = \lceil \frac{s_{k+i-1}}{2} \rceil \geq 2$.

Let us now explain why $k+i$ is a balanced index.
First, by the termination condition we have that $s_{k+i} \leq \frac{f(n_{k+i})}{2n_{k+i}}$ so $2n_{k+i}s_{k+i} \leq f(n_{k+i})$.
Now let us prove the other side of the balanced inequality, namely, $\frac{1}{3}f(\lfloor \frac{n_{k+i}}{3} \rfloor) \leq 2n_{k+i}s_{k+i}$. First, assume that $s_{k+i-1}>\frac{f(n_{k+i-1})}{2n_{k+i-1}}$ (indeed, this must be the case if $i>1$ but not necessarily if $i=1$, as the termination condition requires $i>0$). Then:
\begin{eqnarray*}
2s_{k+i}n_{k+i} \geq  s_{k+i-1}n_{k+i} & > &  \frac{n_{k+i}}{2n_{k+i-1}}f(n_{k+i-1}) \\ & > & f(n_{k+i-1}) \geq f\left(\lfloor \frac{n_{k+i}}{3} \rfloor\right)    
\end{eqnarray*}
where the last inequality follows since $n_{k+i}\leq 3n_{k+i-1}$. 

The remaining case is that $i=1$ and $s_k\leq \frac{f(n_k)}{2n_k}$. But then
$$ \frac{f(n_{k+1})}{4n_{k+1}} = \frac{f(3n_k)}{12n_k}  < \frac{s_k}{2} \leq s_{k+1} \leq \frac{f(n_{k+1})}{2n_{k+1}} $$
(the first inequality follows by Case III assumption and the last inequality follows since we assume that $i=1$ satisfies the termination condition) and from this, balancedness follows:
$$ \frac{1}{3}f\left(\lfloor\frac{n_{k+1}}{3} \rfloor \right) \leq \frac{f(n_{k+1})}{2} < 2n_{k+1}s_{k+1} \leq f(n_{k+1}).$$
As mentioned before, balancedness is a termination condition for the routine of Case III as well.

\bigskip

We close this section by illustrating how the routine described above works for certain examples.

\begin{exmpl} \label{exmpl:8n+1}
Let $f(n)=8n+1$. Then $n_1=1,b=9,s_1=4$. We prove by induction that $n_k=3^{k-1},s_k=4$ and we are always in Case I. Indeed, $\frac{f(3n_k)}{6n_k}=\frac{8\cdot 3^k + 1}{2\cdot 3^k}\in (4,5)$ so $\sqrt{\frac{f(3n_k)}{6n_k}} < s_k < \frac{f(3n_k)}{6n_k}$ and we are in Case I, so $n_{k+1}=3n_k=3^k$ and $s_{k+1}=\lfloor \frac{f(3n_k)}{6n_k} \rfloor = 4$.
\end{exmpl}

\begin{exmpl}
Let $f(1)=8$ and $f(n)=8n+10$ for $n\geq 2$. Then $n_1=1,b=8,s_1=4$ so $\frac{f(3n_1)}{6n_1}=\frac{f(3)}{6}=\frac{34}{6}\in (5,6)$ so $s_1=4$ satisfies $\sqrt{\frac{f(3n_1)}{6n_1}} < s_1 < \frac{f(3n_1)}{6n_1}$ and we start in Case I. Next, $n_2=3,s_2=\lfloor \frac{f(3)}{6} \rfloor = \lfloor \frac{34}{6} \rfloor = 5$ so $\frac{f(3n_2)}{6n_2}=\frac{f(9)}{18}=\frac{82}{18} < 5 = s_2$ and we move to Case III. Thus (for $i=0$ in the sub-routine of Case III), $n_3=(2^2-1)n_2=9,s_3=\lceil \frac{s_2}{2} \rceil = 3$. Moreover, $i=0$ is the first index for which $s_{2+i}\leq \frac{f(n_{2+i})}{2n_{2+i}}$ (indeed, $s_2=5$ and $\frac{f(n_2)}{2n_2}=\frac{f(3)}{6}=\frac{34}{6}$), but we are waiting until the first $i>0$ for which $s_{2+i} \leq \frac{f(n_{2+i})}{2n_{2+i}}$. We compute that $s_3=3\leq \frac{f(n_3)}{2n_3}=\frac{f(9)}{18}=\frac{82}{18}$ so we stop Case III procedure with $n_3=9,s_3=3$. Indeed, $k=3$ is a balanced index as $\frac{1}{3}f(\lfloor \frac{n_3}{3} \rfloor)=\frac{1}{3}f(3) = \frac{34}{3} < 2n_3s_3 = 2\cdot 9\cdot 3 = 54 < f(n_3) = f(9) = 82$.

Compute $\frac{f(3n_3)}{6n_3}=\frac{f(27)}{54}=\frac{8\cdot 27+10}{54} =\frac{226}{54} \in (4,5)$ so $\sqrt{\frac{f(3n_3)}{6n_3}} < s_3=3 < \frac{f(3n_3)}{6n_3}$ so we are back in Case I, and $n_4=3n_3=27,s_4=\lfloor \frac{226}{54} \rfloor = 4$. For $k\geq 4$, $n_k=3^{k-1},s_k=4$ as shown in Example \ref{exmpl:8n+1}, and we remain in Case I from this point indefinitely.
\end{exmpl}

\begin{exmpl}
Let $f(n)=8n+100$. Then $n_1=1,b=108,s_1=54$ so $\frac{f(3n_1)}{6n_1}=\frac{f(3)}{6}=\frac{124}{6} \in (20,21)$ so $s_1=54$ satisfies $s_1>\frac{f(3n_1)}{6n_1}$ and we start in Case III. Notice that already for $i=0$ we have $s_{1+i}\leq \frac{f(n_{1+i})}{2n_{1+i}}$ (namely, $54 \leq \frac{108}{2}$), but we are waiting until the first $i>0$ for which $s_{1+i} \leq \frac{f(n_{1+i})}{2n_{1+i}}$. According to the sub-routine of Case III, $n_2=(2^2-1)n_1=3n_1=3$ and $s_2=\lceil \frac{s_1}{2} \rceil = 27$. We see that $s_2=27 > \frac{f(n_2)}{2n_2}=\frac{f(3)}{6}=\frac{124}{6}$, so we continue. Now $n_3=(2^3-1)n_1=7n_1=7$ and $s_3= \lceil \frac{27}{2} \rceil =14$. Again $s_3=14 > \frac{f(7)}{14}=\frac{156}{14}$, so we continue to $n_4=(2^4-1)n_1=15,s_4=\lceil \frac{14}{2} \rceil =7$. Now $s_4=7 < \frac{f(n_4)}{2n_4} = \frac{220}{30}$, so we exit the sub-routine of Case III here. Notice that indeed $4$ is balanced. We have $\frac{f(3n_4)}{6n_4}=\frac{f(45)}{90}=\frac{460}{90} < 7 =s_4$ so we enter Case III again. Compute $n_5=(2^2-1)n_4=3n_4=45,s_5=\lceil \frac{s_4}{2} \rceil = 4$. We check the termination condition $s_5=4 \leq \frac{f(n_5)}{2n_5}=\frac{f(45)}{90}=\frac{460}{90}$ and see it is fulfilled, so we exit Case III again. Now $\frac{f(3n_5)}{6n_5}=\frac{f(135)}{270} = \frac{1180}{270}\in (4,5)$ so $s_5=4$ satisfies $\sqrt{\frac{f(3n_5)}{6n_5}} < s_5 < \frac{f(3n_5)}{6n_5}$ and we are in Case I. Now by induction for all $k\geq 5$ we have $n_k=3^{k-3}\cdot 5,s_k=4$ and we are in Case I indefinitely. Indeed, $\frac{f(3n_k)}{6n_k}=\frac{f(3^{k-2}\cdot 5)}{3^{k-2}\cdot 10}=\frac{3^{k-2}\cdot 40+100}{3^{k-2}\cdot 10} = 4 + \frac{10}{3^{k-2}}\in (4,5)$ so $\sqrt{\frac{f(3n_k)}{6n_k}} < s_k < \frac{f(3n_k)}{6n_k}$ and $n_{k+1}=3n_k = 3^{k-2} \cdot 5,s_{k+1}=\lfloor \frac{f(3n_k)}{6n_k} \rfloor = 4$.
\end{exmpl}

\begin{exmpl}
Let $f(n)=2^{n+3}$. Then $n_1=1,b=16,s_1=8$ and $\frac{f(3n_1)}{6n_1}=\frac{f(3)}{6}=\frac{64}{6}\in (10,11)$ so $s_1=8$ satisfies $\sqrt{\frac{f(3n_1)}{6n_1}} < s_1 < \frac{f(3n_1)}{6n_1}$ and we start in Case I. Then $n_2=3n_1=3,s_2=\lfloor \frac{f(3n_1)}{6n_1} \rfloor = 10$. Now $\frac{f(3n_2)}{6n_2}=\frac{f(9)}{18} = \frac{2^{12}}{18}$ so $s_2=10 < \sqrt{\frac{f(3n_2)}{6n_2}}$ and we enter Case II. The preparation step gives us $n_3=8n_2=24$, and $$s_3=\min\{\underbrace{\lceil \frac{f(8n_2)}{16n_2} \rceil}_{=\lceil \frac{f(24)}{48} \rceil=\lceil \frac{2^{27}}{48} \rceil},\max\{\underbrace{s_2}_{=10},\underbrace{f(1)}_{=16}\}\}=16.$$ Moreover, $$\frac{1}{3}f(\lfloor \frac{n_3}{3} \rfloor) = \frac{1}{3}\cdot 2^{11} < 2n_3s_3=2\cdot 24 \cdot 16=768 < f(n_3)=2^{27}$$ so $3$ is balanced and we exit Case II. But now $\frac{f(3n_3)}{6n_3}=\frac{f(72)}{144}=\frac{2^{75}}{144}$ and $s_3=16 < \sqrt \frac{f(3n_3)}{6n_3}$, so we enter Case II again. The preparation step gives us $n_4=8n_3=192$, and
$$s_4=\min\{\underbrace{\lceil \frac{f(8n_3)}{16n_3} \rceil}_{=\lceil \frac{f(192)}{384} \rceil=\lceil \frac{2^{195}}{384} \rceil},\max\{\underbrace{s_3}_{=16},\underbrace{f(8)}_{=2^{11}}\}\}=2^{11}.$$
This time $\frac{1}{3}f(\lfloor \frac{n_4}{3} \rfloor)=\frac{1}{3}2^{67} > 2n_4s_4=2\cdot 192 \cdot 2^{11}$, so $4$ is not balanced, and we need to move to either Sub-Case (1) or (2). We compute $g_4=n_4-2n_3=192-2\cdot 24=144,\ \tilde{g}_5=\max\{144,\lceil \frac{192}{4^2} \rceil\}=144$ and $\tilde{n}_5=2\cdot 192+144=528$. Now $\frac{f(\tilde{n}_5)}{2\tilde{n}_5}=\frac{2^{531}}{2\cdot 528} > s_4^2=2^{22}$, so we are in Sub-Case (1) and take $n_5 = \tilde{n}_5 = 528$ and $s_5=s_4^2=2^{22}$. 

We claim that for all $i\geq 1$, the indices $n_{4+i},s_{4+i}$ are computed using Sub-Case (1) (we just showed this for $i=1$). Indeed, $g_{4+i}\geq g_{3+i}\geq g_4=144$ so $\tilde{n}_{4+i}=2n_{3+i}+\tilde{g}_{4+i}\geq 2n_{3+i}+144$ and $\tilde{n}_{4+i}\leq 3n_{3+i}$ and by the induction hypothesis $n_{3+i} = \tilde{n}_{3+i}$, so
$$\frac{f(\tilde{n}_{4+i})}{2\tilde{n}_{4+i}} = \frac{2^{\tilde{n}_{4+i}+3}}{2\tilde{n}_{4+i}} \geq \frac{2^{2 \tilde{n}_{3+i}+147}}{6 \tilde{n}_{3+i}} \geq \left( \frac{2^{\tilde{n}_{3+i}}}{2\tilde{n}_{3+i}} \right)^2 > (s_{2+i}^2)^2=s_{3+i}^2$$
(the last inequality follows from the induction hypothesis that $3+i$ falls into Sub-Case (1)) so we are indeed in Sub-Case (1), $n_{4+i}=\tilde{n}_{4+i},s_{4+i}=s_{3+i}^2$ and so on.
Furthermore, the same inductive process shows that none of these indices is balanced. Indeed, the following computations follow by induction: 
\begin{eqnarray*}
n_{4+i} & = & \tilde{n}_{4+i}\geq 2n_{3+i}\geq 2^i n_4 = 2^i\cdot 192 \\ n_{4+i} & \leq &  3n_{3+i}\leq 3^in_4 = 3^i\cdot 192 \\ s_{4+i} & = & s_{3+i}^2=(2^{11\cdot 2^{i-1}})^2=2^{11\cdot 2^i}
\end{eqnarray*}
so 
\begin{eqnarray*} 
\frac{1}{3}f(\lfloor \frac{n_{4+i}}{3} \rfloor) & \geq & \frac{1}{3}f(2^i\cdot 64) = \frac{1}{3}2^{2^i\cdot 64 + 3}  \geq  2^{2^i\cdot 64}  \\ 
2n_{4+i}s_{4+i} & \leq & 2\cdot 3^i \cdot 192 \cdot 2^{11\cdot 2^i}
\end{eqnarray*}
so $2n_{4+i}s_{4+i} < \frac{1}{3}f(\lfloor \frac{n_{4+i}}{3} \rfloor)$ and $4+i$ is not balanced.
\end{exmpl}

\section{Properties and complexity of $\mathcal{X}$}

Let $L$ be the hereditary closure of $\bigcup_{k=1}^{\infty} X_k$ defined above.

\begin{lem} \label{lem:recurrent}
The formal language $L$ is recurrent. In particular, there exists an infinite recurrent word $\omega\in \Sigma^{\mathbb{N}}$ such that $L$ coincides with the set of finite factors of $\omega$. Furthermore, each $X_k$ can be ordered such that $\omega$ is the limit of the first words in $X_1,X_2,\dots$ under these orders.
\end{lem}

Using this lemma we conclude that $L = \mathcal{L}(\mathcal{X})$ for a recurrent subshift $\mathcal{X}\subseteq \Sigma^{\mathbb{N}}$, the closure of the shift orbit of $\omega$.

\begin{proof}
It suffices to show that for every $k\in \mathbb{N}$ and for every $\alpha,\alpha'\in X_k$ there exists a word $\alpha''\in X_{r}$ (for some $r\geq k$) of the form $\alpha'' = \star \alpha \star \alpha' \star$ (where the symbols $\star$ stand for arbitrary words).

For every $k\in \mathbb{N}$, enumerate $X_k = \{\alpha_1,\dots,\alpha_{s_k}\}$. 
Claim: we may assume that the first $\lceil s_k/2 \rceil$ elements of $X_{k+1}$ are as follows if $s_k$ is even:
$$ X_{k+1} = \{\alpha_1 \star \alpha_2,\ \alpha_3 \star \alpha_4,\ \dots,\ \alpha_{s_k-1} \star \alpha_{s_k},\ \dots\} $$
or, if $s_k$ is odd,
$$ X_{k+1} = \{\alpha_1 \star \alpha_2,\  \alpha_3 \star \alpha_4,\ \dots,\ \alpha_{s_k-2} \star \alpha_{s_k-1},\ \alpha_{s_k} \star \alpha_1,\ \dots\}. $$
This is clearly true for $k$ falling in Case I, since then $X_{k+1}$ contains $\alpha_i 0^{n_k} \alpha_{i+1}$ (taken cyclically modulo $s_k$) for each $1\leq i\leq s_k$. By a similar reason this is true for $k$ falling in the preparation step (the first step) of Case II. If $k$ falls in one of the steps of Sub-Case (1) of Case II, or in the second item of Sub-Case (2), then $X_{k+1}=X_k 0^{\star} X_k$, so the claim is again true (the first and third items reduce to the other cases). If $k$ falls in Case III (not necessarily the first step) then we notice that $X_{k+1}$ consists of the words $\alpha_1 0^\star \alpha_2,\ \dots,\ \alpha_{s_k-1} 0 ^\star \alpha_{s_k}$ if $s_k$ is even and of the words $\alpha_1 0^\star \alpha_2,\ \dots,\ \alpha_{s_k-2} 0 ^\star \alpha_{s_{k}-1},\ \alpha_{s_k},\ \alpha_1$ if $s_k$ is odd. 
A byproduct of this ordering of each $X_{k+1}$ is that we may assume that the first word in $X_k$ is always a prefix of the first word in $X_{k+1}$, for every $k$. Therefore the sequence of first words in $X_k$ as $k\rightarrow \infty$, converges to an infinite word $\omega$.

As a consequence of this claim, in $X_{k+\lceil \log_2{s_k}\rceil}=\{\beta_1,\dots,\beta_m\}$ we will have that the first word is:
$$ \beta_1 = \alpha_1 \star \alpha_2 \star \alpha_3 \star \alpha_4 \star \cdots \star \alpha_{s_k} \star. $$ It thus suffices to show that some further index contains a word of the form $\star \beta_1 \star \beta_1 \star$. If this is true, then it follows that $\omega$ -- the limit of the first words in the sets $X_k$, $k\in \mathbb{N}$ -- has every word in $L$ as a factor. Notice that for each $p\geq k+\lceil \log_2{s_k}\rceil$ we have that the first word in $X_p$ has $\beta_1$ as a prefix.

We first reduce to the case where we have infinitely many balanced indices. The only scenario in which this would not happen, is if we are trapped in a loop of Sub-Case (1) of Case II (see flow diagram), but this case is readily recurrent, as $X_{k+i+1} = X_{k+i} 0^\star X_{k+i}$. Assuming that we have infinitely many balanced indices, observe that we have infinitely many balanced indices falling into Cases I,II; indeed, whenever we enter Case III we must exit it after finitely many steps, as explained in its description. Furthermore, even if we re-enter Case III, we cannot stay there forever: each iteration necessarily decreases $s_k$, so eventually we will get to some $s_{k'} < 4$. If we again fall into Case III, this means that $\frac{f(3n_{k'})}{6n_{k'}} \leq s_{k'} < 4$, contradicting our assumption that $f(n)\geq 8n$ for all $n$.

If $l$ is balanced and belongs to Case I or II then we can assume that, if $s_l$ is even then $X_{l+1}$ starts with: 
\begin{align} \label{sleven}
X_{l+1} = \{ \alpha_1 \star \alpha_2,\alpha_3\star \alpha_4,\dots,\alpha_{s_{l}-1}\star \alpha_{s_l}, \alpha_{s_l} \star \alpha_{1},\dots \}
\end{align}
or, if $s_l$ is odd:
\begin{align} \label{slodd}
X_{l+1} = \{ \alpha_1 \star \alpha_2,\alpha_3\star \alpha_4,\dots,\alpha_{s_{l}-2}\star \alpha_{s_{l}-1}, \alpha_{s_l} \star \alpha_{1},\dots \}.
\end{align}
The new element (the last element written in the above two forms) belongs to these sets since if $l$ is balanced in Cases I or II then we always have $s_l \star s_1 \in X_{l+1}$.

Now fix some balanced index $l>k+\lceil \log_2 s_k\rceil$ which belongs to either Case I or II. By (\ref{sleven}) and (\ref{slodd}) we have that:
$$ X_{l+1} = \{\gamma_1,\dots,\gamma_i,\dots\} $$ where $\gamma_1,\gamma_i$ are factored by $\beta_1$ ($i=\frac{s_l}{2}+1$ if $s_l$ is even and $i=\frac{s_l + 1}{2}$ if $s_l$ is odd). 
Now by the consequence of the claim from the beginning of the proof, in $X_m$ -- for some $m$ sufficiently larger\footnote{In fact, $m=l+1+\lceil \log_2 s_{l+1} \rceil$ suffices.} than $l+1$ -- we will have a word with (disjoint) occurrences of $\gamma_1,\gamma_i$, hence factored by $\beta_1 \star \beta_1$, as desired.
\end{proof}

We now discuss another property of $\mathcal{X}$, which will be useful in analyzing its complexity.

\begin{lem} \label{lem:inf word balanced}
Let $k$ be a balanced index. Then the subshift $\mathcal{X}$ is the closure of the shift orbit of an infinite word of the form:
$$ \rho_1 0^{m_1}\rho_20^{m_2}\rho_3 0^{m_3} \cdots $$
where each $\rho_i\in X_k$ and $m_i\geq n_k$.    
\end{lem}

\begin{proof}
If $k$ is balanced then words in $X_{k+1}$ are obtained by concatenating words from $X_k$ with blocks of zeros of length $\geq n_k$ between them.
In fact, for any $k$ (not necessarily balanced) $X_{k+1}$ is a set of words obtained by concatenating words from $X_k$ with blocks of zeros of length at least $z_k$ between them, where the sequence $(z_k)_{k=1}^{\infty}$ is non-decreasing and satisfies $z_k=n_k$ for every balanced index $k$. This is easily seen if $k$ falls in Case I or III, or in the preparation step of Case II. In Case II in the first step after the preparation step, assuming we are in Sub-Case (1), this follows since $z_{k+1} = \tilde{g}_{k+1}=6n_k \geq n_k = z_k$; afterwards (still in Sub-Case (1)) monotonicity follows from the choice of $\tilde{g}_{k+i+1}$. In Sub-Case (2), the first and the third items follow from Cases I,III respectively; and in the second item, $z_{k+i+1} = M-2n_{k+i} \geq g_{k+i} = z_{k+i}$.

Now fix a balanced index $k$ and let $\rho_1\in X_k$ be the first word with respect to the order given in the proof of Lemma \ref{lem:recurrent}. By Lemma \ref{lem:recurrent}, $\mathcal{X}$ has a dense shift orbit of a recurrent word $\omega$, which is the limit of the first words in $X_k,X_{k+1},\dots$; in other words, the length-$n_{l}$ prefix of $\omega$ is the first word of $X_{l}$, for each $l$. Decomposing each such prefix for $l>k$ as a concatenation of words from $X_k$ and blocks of zeros, we obtain by the above arguments that these blocks of zeros must be of lengths $\geq n_k$. The result follows.
\end{proof}

We now turn to analyzing the complexity function $p_\mathcal{X}$ of $\mathcal{X}$.

\begin{cor} \label{cor:balanced complexity}
Let $k$ be a balanced index. Then $p_\mathcal{X}(2n_k) \geq n_k s_k \geq \frac{1}{6}f\left(\lfloor \frac{n_k}{3} \rfloor\right)$ and $p_\mathcal{X}(n_k) \leq 3n_k s_k \leq \frac{3}{2}f(n_k)$.
\end{cor}

\begin{proof}
By Lemma \ref{lem:inf word balanced}, we have that the words $0^j \alpha 0^{n_k - j},\ j=0,\dots,n_k,\ \alpha\in X_k$ are all distinct since $\alpha$ cannot start or end with $0$. Therefore:
$$ p_\mathcal{X}(2n_k) \geq n_k s_k \geq \frac{1}{6} f\left(\lfloor \frac{n_k}{3} \rfloor\right) $$ where the last inequality follows since $k$ is balanced.

Now consider an arbitrary length-$n_k$ factor $\alpha$ of $\mathcal{X}$, or equivalently, of $\omega$. By Lemma \ref{lem:inf word balanced}, we know that $\alpha$ takes the form $0^j \alpha'$ for some $0 \leq j \leq n_k$ and $\alpha'$ a length-$n_k-j$ prefix of a word from $X_k$, or $\alpha'' 0^j$ for some $0\leq j\leq n_k$ and $\alpha''$ a suffix of a word from $X_k$. Altogether, we have at most $(2n_k+1)s_k$ options for $\alpha$, so:
$$ p_\mathcal{X}(n_k) \leq (2n_k + 1)s_k \leq 3n_k s_k \leq \frac{3}{2} f(n_k) $$
where the last inequality again follows since $k$ is balanced.
\end{proof}

\begin{lem} \label{lem:complexity bounds}
The complexity function of $\mathcal{X}$ satisfies $p_{\mathcal{X}} \sim f$.
\end{lem}

\begin{proof}
It suffices to prove that $p_\mathcal{X}(Cn_k)\geq f(\lfloor c n_k \rfloor)$ and that $p_\mathcal{X}(n_k)\leq Df(n_k)$ for some constants $C,c,D>0$. That this suffices follows from Lemma \ref{lem:enough n_k}.
These inequalities hold for balanced indices by Corollary \ref{cor:balanced complexity}. 
It remains to establish such inequalities for indices in either Case II or III.

Look at $k$ balanced, which falls in Case II. We may assume that: $$ s_{k+1} = \max \{ s_k , f(\lfloor n_k/3 \rfloor)\}, $$ since otherwise we are again in a balanced position. Now for every $i\geq 1$, we may assume that we are in Sub-Case (1) (this is because Sub-Case (2) holds for at most one isolated index, and then we are in a balanced position; we may omit these sporadic indices, and the indices immediately before them again by Lemma \ref{lem:enough n_k}):
\begin{align} \label{eqn:7}
p_\mathcal{X}(n_{k+i+1}) \geq s_{k+i+1} = s_{k+1}^{2^i}\geq f(\lfloor n_k/3 \rfloor)^{2^i} \geq f(2^i\lfloor n_k/3 \rfloor)    
\end{align}
We next claim that $n_{k+i+1} \leq C \cdot 2^i \cdot n_k$ for some constant $C>0$.

Recall that in Sub-Case (1) we have, by definition:
$$ \frac{n_{k+i+1}}{2n_{k+i}} = 1 + \frac{g_{k+i+1}}{2n_{k+i}} $$
we have:
\begin{align} \label{eqn:consecutive}
 \frac{n_{k+i+1}}{2^i n_{k+1}} = \left(1+\frac{g_{k+2}}{2n_{k+1}}\right)\left(1+\frac{g_{k+3}}{2n_{k+2}}\right)\cdots\left(1+\frac{g_{k+i+1}}{2n_{k+i}}\right).
\end{align}

Recall that $g_{k+1} = 6n_k,\ n_{k+1}=8n_k$. There are two cases to consider; the first is that $g_{k+i+1}=g_{k+i}=\cdots = g_{k+1}=6n_k$. In this case,
$$ (\ref{eqn:consecutive}) \leq \left(1+\frac{6n_k}{16n_k}\right)\left(1+\frac{6n_k}{32n_k}\right)\cdots = \prod_{i=0}^{\infty} \left(1+\frac{3}{2^{i+3}}\right) \leq 2. $$
The second case is that for some $i_0$ we have: $6n_k = g_{k+1} = \cdots = g_{k+i_0}$ and $g_{k+i_0+1} = \lceil \frac{n_{k+i_0}}{(k+i_0)^2} \rceil$.
We claim that for $r\geq 1$, we have that: $$ g_{k+i_0+r} = \lceil \frac{n_{k+i_0+r-1}}{(k+i_0+r-1)^2} \rceil. $$ Indeed, once $g_{t+1} = \lceil \frac{n_t}{t^2} \rceil$, the options for $g_{t+2}$ are $g_t=\lceil \frac{n_t}{t^2} \rceil,\ \lceil \frac{n_{t+1}}{(t+1)^2} \rceil$. Recall that $n_t\geq 2^t$. 
The ratio of these two options is (we may assume that $t$ is sufficiently large, e.g. $t\geq 7$, so that $t^2<\frac{1}{2}2^t$ and that $\left(\frac{t}{t+1}\right)^2\geq \frac{3}{4}$):
\begin{eqnarray*}
\lceil \frac{n_{t+1}}{(t+1)^2} \rceil / \lceil \frac{n_t}{t^2} \rceil & \geq & \frac{n_{t+1}}{(t+1)^2} \cdot \frac{t^2}{n_t+t^2} \\ & \geq & \frac{n_{t+1}}{(t+1)^2} \cdot \frac{t^2}{\frac{3}{2}n_t} \\ & \geq &  \frac{2n_t}{\frac{3}{2}n_t}\cdot \frac{3}{4} \geq 1
\end{eqnarray*}
and therefore $g_{t+1} = \lceil \frac{n_{t+1}}{(t+1)^2} \rceil$ for $t\geq i_0$. Furthermore, for such indices 
$g_{t+1} \leq \frac{n_{t+1}}{(t+1)^2}+1$ so
$\frac{g_{t+1}}{2n_t} \leq \frac{n_{t+1}/n_t}{2(t+1)^2} + \frac{1}{2n_t} \leq \frac{3}{2(t+1)^2}+2^{-t}$. 
Therefore:
\begin{eqnarray}
(\ref{eqn:consecutive}) \leq \prod_{i=0}^{\infty} \left(1+\frac{3}{2^{i+3}}\right) \cdot \prod_{i=1}^{\infty} \left(1+\frac{3}{2(k+i+1)^2}+2^{-(k+i)}\right) < K <\infty
\end{eqnarray}
for some universal constant $K>0$. It now follows that in (\ref{eqn:7}), $n_{k+i+1} \leq K\cdot 2^i\cdot n_k$ and therefore:
\begin{align} \label{lower bound for Case II}
p_\mathcal{X}(n_{k+i+1}) \geq f(2^i\lfloor n_k/3 \rfloor ) \geq f(\lfloor \frac{1}{6K} n_{k+i+1} \rfloor)
\end{align}
and so the desired lower bound applies for $k$ in Case II.

We now turn to establish an upper bound on $p_\mathcal{X}(n_{k+i})$. By Lemma \ref{lem:inf word balanced}, every length-$n_{k+i}$ factor of $\mathcal{X}$ takes one of the following forms:
\begin{align} \label{cases pre suf}
\alpha' 0^p, & \ & \alpha'\ \text{a suffix of  length}\ n_{k+i} - p\ \text{of a word from}\ X_{k+i} \\
0^p\alpha'', & \ & \alpha''\ \text{a prefix of  length}\ n_{k+i} - p\ \text{of a word from}\ X_{k+i} \nonumber \\
\alpha' 0^p \alpha'', & \ & \alpha'\ \text{a suffix of  length}\ q\leq n_{k+i} - p\ \text{of a word from}\ X_{k+i} \nonumber \\
 & \ & \alpha''\ \text{a prefix of  length}\ n_{k+i} - p - q\ \text{of a word from}\ X_{k+i}. \nonumber 
 \end{align}

The number $p$ of zeros in the third case of (\ref{cases pre suf}) block satisfies: $$ p\in \{g_{k+i},g_{k+i+1},\dots\} \cap [1,n_{k+i}]. $$ We claim that this set is rather small. As mentioned before, once $g_{t+1}=\lceil \frac{n_t}{t^2} \rceil$ for $t=t_0$ (otherwise, there is only one option for such $p$) we will have this formula for all $t_0 \leq t\leq k+i$. Since the numerator $n_t$ at least doubles itself ($n_{t+1}\geq 2n_t$) and the denominator is quadratic, it follows that there exists some constant $K' > 0$ such that $g_{k+i+\lfloor K'\log_2 (k+i) \rfloor} > n_{k+i}$, so $p$ in the third case of (\ref{cases pre suf}) has at most $K'\log_2 (k+i)$ options. Therefore:
\begin{eqnarray*}
p_\mathcal{X}(n_{k+i}) & \leq &  \underbrace{(2n_{k+i}+1)s_{k+i}}_{\substack{\text{the first two cases} \\ \text{of (\ref{cases pre suf})}}}+\underbrace{(n_{k+i}+1)}_{\substack{\text{options for the length} \\ \text{of $\alpha'$ in the third case}}} \underbrace{K'\log_2(k+i)}_{\substack{\text{options for the length} \\ \text{of the gap of zeros}}} \underbrace{s_{k+i}^2}_{\substack{\text{the participating} \\ \text{prefix and suffix} }} \\
& \leq & 3n_{k+i}s_{k+i} + 2K'n_{k+i}\log_2 (k+i) s_{k+i}^2 \\ & \leq & (2K'+3)n_{k+i}s_{k+i}^4 \leq  (2K'+3)n_{k+i+1}s_{k+i+1}^2 \\ & \leq & (2K'+3)f(n_{k+i+2}) \leq (2K'+3)f(9n_{k+i}).
\end{eqnarray*} 
We used the bound $\log_2(k+i)\leq s_{k+i}^2=s_{k+i+1}$, which holds since $s_{k+i+1}=s_{k+1}^{2^i}\geq f(\lfloor n_k/3 \rfloor)^{2^i} \geq n_k^{2^i} \geq 2^{k2^i} \geq \log_2(k+i)$. We also used the assumption of Sub-Case (1) for $k+i+1$ and $k+i+2$ -- we may assume this, since as we noted, we may neglect Sub-Case (2) indices in this analysis.

Finally, let us analyze the complexity for indices falling in Case III. Let $k$ be a balanced index starting Case III and let $j \geq 1$ be such that $k+j$ is still in the routine of Case III. Then $p_\mathcal{X}(2n_k) \geq s_k n_k$ by Corollary \ref{cor:balanced complexity}. Case III cannot hold forever, so let $k+i$ be the next balanced index (with $j \leq i$). 
There are $s_{k+i}=\lceil \frac{s_k}{2^i} \rceil$ words in $X_{k+i}$ and $n_{k+i}=(2^{i+1}-1)n_k$, so $s_{k+i}n_{k+i} \leq 4 s_k n_k$.
Since $k+i$ is balanced, we know by Corollary \ref{cor:balanced complexity} that 
$p_\mathcal{X}(n_{k+i}) \leq 3n_{k+i}s_{k+i}\leq 12 s_k n_k$, and so:
\begin{eqnarray*}
\frac{1}{24}f(\lfloor n_{k+j}/3 \rfloor) & \leq &  \frac{1}{24} f(\lfloor n_{k+i}/3 \rfloor) \\ & \leq & \frac{1}{4} s_{k+i} n_{k+i} \leq s_k n_k \\ &  \leq & p_\mathcal{X}(2n_k) \leq p_\mathcal{X}(2n_{k+j})    
\end{eqnarray*}
and: 
$$ p_\mathcal{X}(n_{k+j}) \leq p_\mathcal{X}(n_{k+i}) \leq 12 s_k n_k \leq 6f(n_k) \leq 6 f(n_{k+j}). $$
It follows that $p_\mathcal{X} \sim f$.
\end{proof}

\begin{proof}[{Proof of Theorem \ref{thm:main}}]
Let $f$ be a function satisfying the assumptions of the theorem. We follow the construction in Section \ref{sec:construction}. By Lemma \ref{lem:recurrent}, $\mathcal{X}$ is recurrent. By Lemma \ref{lem:complexity bounds}, $p_\mathcal{X} \sim f$.
\end{proof}

\begin{rem}
The statement of Theorem \ref{thm:main} promises to realize any function $f$ satisfying the conditions mentioned therein as the complexity function of a subshift, up to asymptotic equivalence. In fact, the (inner and outer) constants of the asymptotic equivalence are bounded universally over all functions, as can be seen by tracing these constants along the proof.
\end{rem}

\section{Minimal subshifts}

In this section we prove Proposition \ref{prop:minimal}. Namely, we prove that a non-decreasing, unbounded, submultiplicative function $f\colon \mathbb{N}\rightarrow \mathbb{N}$ is asymptotically equivalent to the complexity function of a minimal subshift, up to a linear error factor.  We assume that $0$ is not in the range of $f$, and indeed, the complexity function of any subshift cannot assume this value. 
We start by illustrating the proof strategy. 

\subsection{Proof strategy}
Let $X$ be a finite alphabet, $|X|\geq 2$. 
Given sets $A,B$ of words over the alphabet $X$, we denote $AB=\{\alpha\beta\ |\ \alpha\in A, \beta\in B\}$. We construct, for each $n\geq 0$, sets $W(2^n),C(2^n)\subseteq X^{2^n}$ such that $W(1)=X$ and $C(2^n)\subseteq W(2^n)$ and $W(2^{n+1})=W(2^n)C(2^n)$.
This construction, introduced by Smoktunowicz and Bartholdi in \cite{SB} (modified in \cite{BGSel} for recurrent words and in \cite{BGIsr} to uniformly recurrent words under additional assumptions), gives a collection of infinite words $W(1)C(1)C(2)C(2^2)\cdots\subseteq X^{\infty}$, whose shift closure we denote by $\mathcal{X}$. Let $L$ denote the set of finite factors of $\mathcal{X}$.

The first main task is to ensure that $\mathcal{X}$ is minimal. Toward this end, we hold --- throughout the construction of $W(2^n),C(2^n)$ --- a dynamical sequence of words, which can be thought of as `virtual memory' state, and that eventually every word from $\bigcup_{n=1}^{\infty} W(2^n)$ appears in it at some point. In selected steps (when $n\in \{\mu_1,\mu_2,\dots\}$, a carefully chosen sequence), we produce $C(2^n)$ such that the first word in the current state of the memory is the prefix of all words in $C(2^n)$. This, together with the fractal nature of the words in $\mathcal{X}$, ensures uniform recurrence. 
The second goal is to ensure that the complexity of $\mathcal{X}$ is  bounded by $f(n)$ from below and by $n f(n)$ from above. 

We start with two auxiliary lemmas: Lemma \ref{lem:SB} enables to control the complexity of the resulting subshift; it is taken from \cite{SB} and brought here for completeness. Lemma \ref{lem:subexp} shows that we can indeed pick enough $\mu_n$'s as above to afford the construction. Finally, we prove Proposition \ref{prop:minimal}.

\subsection{Auxiliary lemmas}

We start with the following auxiliary lemma on submultiplicative functions.

\begin{lem} \label{lem:subexp}
Let $f\colon \mathbb{N}\rightarrow \mathbb{N}$ be a subexponential, non-decreasing function satisfying  $\lim_{n\to\infty}f(n)=\infty$ and $f(2n)\leq f(n)^2$ for every $n\in \mathbb{N}$. We assume that $0$ is not in the range of $f$.
Then there exists a sequence $1<\mu_1<\mu_2<\cdots$ such that for every $n$ we have $\mu_n > n$ and
\begin{align} \label{eqn:space} \lceil \frac{f(2^{\mu_n+1})}{f(2^{\mu_n})} \rceil \leq \frac{f(2^{\mu_n})}{4f(2^{n})}. \end{align}
\end{lem}

\begin{proof}
It suffices to find a sequence $(a_k)_{k=1}^{\infty}$ for which
$$ \lim_{k\rightarrow \infty} \frac{f(2^{a_k+1})}{f(2^{a_k})^2} = 0 $$
since then, given $n$ and assuming that $1<\mu_1<\cdots<\mu_n$ have been defined, we can take some $a_k> \mu_n$ such that $\frac{f(2^{a_k+1})}{f(2^{a_k})^2}<\frac{1}{8f(2^n)}$ and therefore $\lceil \frac{f(2^{a_k+1})}{f(2^{a_k})} \rceil \leq 2\frac{f(2^{a_k+1})}{f(2^{a_k})} < \frac{f(2^{a_k})}{4f(2^n)}$. Thus, taking $\mu_{n+1}:=a_k$, Inequality \eqref{eqn:space} holds. 
Assume to the contrary that such a sequence does not exist; then: $$ \varepsilon := \inf_{n\in \mathbb{N}} \frac{f(2^{n+1})}{f(2^n)^2} > 0. $$ Fix $n_0\in \mathbb{N}$ such that $f(2^{n_0})>1/\varepsilon$, which is possible by the assumption that $\lim_{n\rightarrow \infty} f(n)=\infty$. Then, for every $N\in \mathbb{N}$, we have that
\begin{align} f(2^{N}) \geq \varepsilon f(2^{N-1})^2 \geq \varepsilon^3 f(2^{N-2})^{2^2} \geq \cdots \geq \varepsilon^{2^{N-n_0}-1} f(2^{n_0})^{2^{N-n_0}}. \nonumber \end{align} Thus
\begin{eqnarray*} f(2^N)^{\frac{1}{2^N}} & \geq & \varepsilon^{\frac{2^{N-n_0}-1}{2^N}} f(2^{n_0})^{2^{-n_0}} \\ & = & \varepsilon^{2^{-n_0} - 2^{-N}} f(2^{n_0})^{2^{-n_0}} \xrightarrow{N\rightarrow \infty} (\varepsilon f(2^{n_0}))^{2^{-n_0}} > 1, \nonumber \end{eqnarray*}
contradicting the assumption that $f$ is subexponential (recall that for a subexponential function $f$ we have that $\lim_{N\rightarrow \infty} f(N)^{1/N} = 1$).
\end{proof}

We now recall a result from \cite{SB} describing the finite subwords that arise from the aforementioned construction of $\mathcal{X}$ in an efficient way.

\begin{lem}\label{lem:SB}{\cite[Lemma 6.3]{SB}} Suppose that $W(2^n),C(2^n)\subseteq X^{2^n}$ satisfy $W(1)=X,C(2^n)\subseteq W(2^n),W(2^{n+1})=W(2^n)C(2^n)$ and let $L$ be the hereditary language of all finite subwords of $\bigcup_{n=1}^{\infty} W(2^n)$. 
Let $w\in L$ be of length $2^m$. Then $w$ is a factor of a word from $W(2^m)C(2^m)\cup C(2^m)W(2^m)$.
\end{lem}

\begin{proof}
For $m=0$, the statement is evident.
Assume that the lemma does not hold. Let $m$ be the smallest positive integer such that there exists $w\in L$ of length $2^m$, not factoring any word from $W(2^m)C(2^m)\cup C(2^m)W(2^m)$. Notice that since $w\in L$, there exists some $n$ such that $w$ factors some $u\in W(2^n)C(2^n)\cup C(2^n)W(2^n)$. We may further assume that $n$ is smallest possible, and notice that by the assumption toward contradiction, $n>m$. 

Recall that for each $n\geq 1$ we have $W(2^n)=W(2^{n-1})C(2^{n-1})$ and $C(2^n)\subseteq W(2^n)$ by construction, so $$W(2^n)C(2^n)\cup C(2^n)W(2^n) \subseteq W(2^{n-1})C(2^{n-1})W(2^{n-1})C(2^{n-1})$$
and we can write $$u=u_1u_2u_3u_4,\ \ \ u_1,u_3\in W(2^{n-1}),\ u_2,u_4\in C(2^{n-1}).$$
If $w$ factors either $u_1u_2$ or $u_3u_4$ then $w$ factors a word from $W(2^{n-1})C(2^{n-1})$, and if $w$ factors $u_2u_3$ then $w$ factors a word from $C(2^{n-1})W(2^{n-1})$; either way, we obtain a contradiction to the minimality of $n$. It follows that $w$ factors $u_1u_2u_3u_4$ but not any of $u_1u_2,u_2u_3,u_3u_4$, so $|w|>2^{n-1}$. By assumption, $n>m$, so $|w|>2^m$, contradicting the way we chose $w$.
\end{proof}

\subsection{Proof of Proposition \ref{prop:minimal}}

\begin{proof}[{Proof of Proposition \ref{prop:minimal}}]
Let $f\colon \mathbb{N}\rightarrow \mathbb{N}$ be a non-decreasing function satisfying $\lim_{n\rightarrow \infty} f(n)=\infty$ and $f(2n)\leq f(n)^2$ for every $n\in \mathbb{N}$. We may assume that $f$ is subexponential, for otherwise it is equivalent to the complexity function of any minimal subshift of positive entropy (see e.g. \cite{Grill}). 
By Lemma \ref{lem:subexp}, there exists a sequence $1<\mu_1<\mu_2<\cdots$ such that for every $n\geq 1$, we have $\mu_n>n$ and
\begin{align} \label{eqn:reallygoestozero}
\lceil \frac{f(2^{\mu_n+1})}{f(2^{\mu_n})} \rceil \leq \frac{f(2^{\mu_n})}{4f(2^n)}.
\end{align}

Let $X$ be a set of cardinality $b:=f(1)$.
As in \cite[Subsection 3.2]{BGSel}, define a sequence of integers $c_{2^n}$ for $n\geq 0$ as follows. First, $c_1=\lceil \frac{f(2)}{f(1)} \rceil$. For $n\geq 1$:
\begin{align} \label{eqn:construction}
c_{2^n} = \begin{cases}
    \lceil \frac{f(2^{n+1})}{f(2^n)} \rceil & \text{if}\  bc_1c_2c_{2^2}\cdots c_{2^{n-1}}<2f(2^n) \\
      \lfloor \frac{f(2^{n+1})}{f(2^n)} \rfloor & \text{if}\  bc_1c_2c_{2^2}\cdots c_{2^{n-1}}\geq 2f(2^n).
\end{cases}
\end{align}
As in \cite[Lemma 3.3]{BGSel}, we have
\begin{lem}[{\cite[Lemma 3.3]{BGSel}}] \label{lem:3.3}
We have $f(2^{n+1}) \leq bc_1c_2c_{2^2}\cdots c_{2^n} \leq 4 f(2^{n+1})$ for all $n\geq 0$.
\end{lem} 
For the reader's convenience, we repeat the proof.
\begin{proof}[{Proof of Lemma \ref{lem:3.3}}]
For $n=0$, we have $$f(2)\leq f(1)\lceil \frac{f(2)}{f(1)} \rceil = bc_1 \leq f(1)\left(\frac{f(2)}{f(1)}+1\right)=f(1)+f(2)\leq 4f(2).$$ Assume that for some $n\geq 1$ we have $f(2^{n}) \leq bc_1c_2c_{2^2}\cdots c_{2^{n-1}} \leq 4 f(2^{n})$ and let us prove $f(2^{n+1}) \leq bc_1c_2c_{2^2}\cdots c_{2^{n}} \leq 4 f(2^{n+1})$.

\smallskip

\noindent \emph{Case I: $bc_1c_2c_{2^2}\cdots c_{2^{n-1}} < 2f(2^n)$.} Then $c_{2^n}=\lceil \frac{f(2^{n+1})}{f(2^n)} \rceil$, so:
$$f(2^{n+1}) = f(2^n)\cdot \frac{f(2^{n+1})}{f(2^n)} \leq bc_1c_2c_{2^2}\cdots c_{2^{n-1}}c_{2^n}$$
and
$$ bc_1c_2c_{2^2}\cdots c_{2^{n-1}}c_{2^n} < 2f(2^n)\left(\frac{f(2^{n+1})}{f(2^n)}+1\right)=2f(2^{n+1})+2f(2^n) \leq 4f(2^{n+1}),$$
where the first inequality follows from the case assumption.

\noindent \emph{Case II: $bc_1c_2c_{2^2}\cdots c_{2^{n-1}} \geq 2f(2^n)$.} Then $c_{2^n}=\lfloor \frac{f(2^{n+1})}{f(2^n)} \rfloor$, so:
$$bc_1c_2c_{2^2}\cdots c_{2^{n-1}}c_{2^n} \leq 4f(2^n)\cdot \frac{f(2^{n+1})}{f(2^n)}=4f(2^{n+1}).$$
As for the other inequality, if furthermore $f(2^{n+1})\geq 2f(2^n)$ then
\begin{eqnarray*}
bc_1c_2c_{2^2}\cdots c_{2^{n-1}}c_{2^n} &  \geq &  bc_1c_2c_{2^2}\cdots c_{2^{n-1}}\left(\frac{f(2^{n+1})}{f(2^n)}-1\right) \\ & \geq &  2f(2^n)\left(\frac{f(2^{n+1})}{f(2^n)}-1\right) = 2f(2^{n+1})-2f(2^n)\geq f(2^{n+1})
\end{eqnarray*}
and if $f(2^{n+1}) < 2f(2^n)$ then
$$bc_1c_2c_{2^2}\cdots c_{2^{n-1}}c_{2^n} \geq bc_1c_2c_{2^2}\cdots c_{2^{n-1}} \geq 2f(2^n) \geq f(2^{n+1}).$$
The proof of the lemma is completed.
\end{proof}

Therefore, for all $n\geq 1$,
\begin{align} \label{eqn:c2n}
    c_{2^n} \leq \lceil \frac{f(2^{n+1})}{f(2^n)}\rceil \leq \lceil f(2^n) \rceil = f(2^n) \leq bc_1c_2c_{2^2}\cdots c_{2^{n-1}}
\end{align}
where the first inequality is by the definition of $c_{2^n}$, the second follows from the submultiplicativity assumption on $f$, and the last one follows from Lemma \ref{lem:3.3}.

\bigskip

We next describe a routine constructing, for each $n\geq 0$, non-empty sets $W(2^n),C(2^n),U(2^n)$ such that
\begin{enumerate}
    \item $W(1)=X$, $C(2^n) \subseteq W(2^n) \subseteq X^{2^n}$ and $W(2^{n+1})=W(2^n)C(2^n)$; \item $|C(2^n)|=c_{2^n},\ |W(2^n)|=bc_1c_2c_{2^2}\cdots c_{2^{n-1}}$; \item Each $U(2^n)$ is a finite sequence of words (not necessarily of the same length) from $\bigcup_{i=1}^{n} W(2^i)$.
    \end{enumerate}

\bigskip

Take $W(1)=X$ and $U(1)=W(1)$, arbitrarily ordered. Notice that $|W(1)|=b$ obeys the rule $|W(2^n)|=bc_1c_2c_{2^2}\cdots c_{2^{n-1}}$ (the product of the $c$'s is empty since $n=0$).

\bigskip

Given $n\geq 0$, assume that $W(2^i),U(2^i)$ have been defined for all $0 \leq i\leq n$ and that $C(2^i)$ has been defined for all $0\leq i < n$. Write $U(2^n) = (u_1,\dots,u_r)$. We explain how to construct $C(2^n)$ and $W(2^{n+1}),U(2^{n+1})$ in each one of the following two cases. 

By assumption, $W(2^n)=W(1)C(1)C(2)C(2^2)\cdots C(2^{n-1})$ has already been constructed and is of cardinality $bc_1c_2c_{2^2}\cdots c_{2^{n-1}}$.

\bigskip

\noindent \emph{Case I: $n=\mu_m$ for some $m$.} Denote $p:=\log_2 |u_1|$, namely, $u_1\in W(2^p)$, and check if $p\leq m$. Suppose that this inequality holds. Then
\begin{align} \label{eqn:ggg}
\lceil \frac{f(2^{n+1})}{f(2^n)}\rceil = \lceil \frac{f(2^{\mu_m+1})}{f(2^{\mu_m})}\rceil \leq \frac{f(2^{\mu_m})}{4f(2^m)} \leq \frac{f(2^n)}{4f(2^p)}
\end{align}
(the middle inequality follows from (\ref{eqn:reallygoestozero}).)
Consider the subset $$Y:=u_1 C(2^p)C(2^{p+1})\cdots C(2^{n-1})\subseteq W(2^n).$$ Notice that
\begin{align}
    |Y| = |C(2^p)||C(2^{p+1})|\cdots |C(2^{n-1})| & = & c_{2^p}c_{2^{p+1}}\cdots c_{2^{n-1}} \\ & = & \frac{bc_1c_2c_{2^2}\cdots c_{2^{n-1}}}{bc_1c_2c_{2^2}\cdots c_{2^{p-1}}} \nonumber \\ 
    & \overset{\text{Lemma}\ \ref{lem:3.3}}{\geq} & \frac{f(2^n)}{4 f(2^p)} \nonumber \\ & \overset{(\ref{eqn:ggg})}{\geq} & \lceil \frac{f(2^{n+1})}{f(2^n)}\rceil \overset{(\ref{eqn:construction})}{\geq} c_{2^n} \nonumber
\end{align}
so we can pick $C(2^n)$ to be an arbitrary subset of size $c_{2^n}$ of $Y$.

We take $W(2^{n+1})=W(2^n)C(2^n)$.

Finally, $U(2^{n+1})$ is obtained from $U(2^n)=(u_1,\dots,u_r)$ by (i) omitting $u_1$, and then (ii) appending $W(2^{n+1})$, ordered arbitrarily, to the end of $U(2^n)$. Thus $U(2^{n+1})$ looks like $(u_2,u_3,\dots,u_r,w_1,\dots,w_s)$, where $W(2^{n+1})=\{w_1,\dots,w_s\}$.

\smallskip

If $p>m$, we move to Case II.

\bigskip

\noindent \emph{Case II.} If either $n\notin \{\mu_1,\mu_2,\dots\}$ or $n=\mu_m$ but $u_1 \in W(2^p)$ for $p>m$ (this is the complement of Case I), we simply take $C(2^n)$ to be an arbitrary subset of size $c_{2^n}$ of $W(2^n)$. This is possible by (\ref{eqn:c2n}), which ensures that $c_{2^n} \leq bc_1c_2c_{2^2}\cdots c_{2^{n-1}}=|W(2^n)|$.

We take $W(2^{n+1})=W(2^n)C(2^n)$. 

Finally, $U(2^{n+1})$ is obtained from $U(2^n)=(u_1,\dots,u_r)$ by appending $W(2^{n+1})$, ordered arbitrarily, to its end. Thus $U(2^{n+1})$ looks like $(u_1,u_2,\dots,u_r,w_1,w_2,\dots,w_s)$, where $W(2^{n+1})=\{w_1,\dots,w_s\}$.

\smallskip

The above routine thus provides us with a system of sets $\{W(2^n),C(2^n),U(2^n)\}_{n=0}^{\infty}$ as desired.

\begin{rem} \label{rem:sometime}
The sets $U(2^n)$ can become arbitrarily long as $n$ grows. However, notice that the first word in $U(2^n)$ will be omitted at some point. Indeed, if $u_1\in W(2^p)$ then, for the first $m$ such that $\mu_m\geq n$ and $m\geq p$, at the stage of constructing $C(2^{\mu_m})$, we omit $u_1$ when defining $U(2^{\mu_m + 1})$ --- as per Case I above.
\end{rem}

\bigskip

Consider the subshift $\mathcal{X}$ generated by the following set of right-infinite words:
$$ W(1)C(1)C(2)C(2^2)\cdots \subseteq X^{\infty}. $$
We aim to prove that $\mathcal{X}$ is minimal (hence defined by a single uniformly recurrent word), and that $f(n)\preceq p_\mathcal{X}(n) \preceq nf(n)$.

\bigskip

\noindent \emph{Minimality.} 
We aim to prove that for every finite factor $u$ of $\mathcal{X}$ there exists some $N_u\geq 1$ such that every factor of $\mathcal{X}$ of length at least $N_u$, is factored by $u$. Evidently, it suffices to prove this claim for every $u\in \bigcup_{n=1}^{\infty} W(2^n)$. Fix $p$ such that $u\in W(2^p)$.

By construction, $u$ belongs to $U(2^{p})$. Furthermore, by Remark \ref{rem:sometime}, for some $n > p$ we have that in $U(2^n)$, the first word is $u_1=u$. Assume that $n$ is the minimal such index. Let $m$ be the smallest integer such that $\mu_m\geq n$ and $m\geq p$. Then by the above routine, Case I, we have that $C(2^{\mu_m})=u_1C(2^p)C(2^{p+1})\cdots C(2^{\mu_m - 1})=uC(2^p)C(2^{p+1})\cdots C(2^{\mu_m - 1})$.

Let $w$ be any factor of $\mathcal{X}$ of length $2^{\mu_m + 2}$. By Lemma \ref{lem:SB}, $w$ factors a word from 
\begin{eqnarray} \label{eqn:WC8} 
& & W(2^{\mu_m+2})C(2^{\mu_m+2})\cup C(2^{\mu_m+2})W(2^{\mu_m+2}) \subseteq \nonumber \\ & & W(2^{\mu_m+2})W(2^{\mu_m+2}) \subseteq \nonumber \\ & &
W(2^{\mu_m+1})C(2^{\mu_m+1})W(2^{\mu_m+1})C(2^{\mu_m+1}) \subseteq \nonumber \\ & & 
W(2^{\mu_m+1})W(2^{\mu_m+1})W(2^{\mu_m+1})W(2^{\mu_m+1}) \subseteq \nonumber \\ & & 
W(2^{\mu_m})C(2^{\mu_m})W(2^{\mu_m})C(2^{\mu_m})W(2^{\mu_m})C(2^{\mu_m})W(2^{\mu_m})C(2^{\mu_m})
\end{eqnarray}
and every length-$2^{\mu_m+2}$ factor of a word from (\ref{eqn:WC8}) contains a factor from $C(2^{\mu_m})$, which in turn contains $u$ as factor.

\bigskip

\noindent \emph{Complexity.} 
By Lemma \ref{lem:SB}, every length-$2^n$ factor of $\mathcal{X}$ is a subword of $W(2^n)C(2^n)\cup C(2^n)W(2^n)$, so: 
\begin{eqnarray*} p_\mathcal{X}(2^n) & \leq & (2^n+1) \cdot |W(2^n)C(2^n)\cup C(2^n)W(2^n)| \\ 
& \leq & (2^n+1)\cdot 2|W(2^n)||C(2^n)| \\ & = & 2(2^n+1)\cdot bc_1c_2c_{2^2}\cdots c_{2^{n-1}} \cdot c_{2^n} \\ & \overset{\text{Lemma}\ \ref{lem:3.3}}{\leq} & 2(2^n+1) \cdot 4f(2^{n+1}).
\end{eqnarray*} 
For the other side of the inequality, notice that every word in $W(2^n)$ is a $2^n$-factor of $\mathcal{X}$, hence $$p_{\mathcal{X}}(2^n)\geq |W(2^n)|=bc_1c_2c_{2^2}\cdots c_{2^{n-1}} \overset{\text{Lemma}\ \ref{lem:3.3}}{\geq} f(2^n).$$ We thus proved that $$f(n)\preceq p_\mathcal{X}(n)\preceq n f(n),$$
(recall Lemma \ref{lem:enough n_k}) as claimed.
\end{proof}

We conclude with the following application to the growth functions of algebras. Notice that the first examples of finitely generated simple groups of subexponential growth were constructed by Nekrashevych in \cite{Nekrashevych simple}.

Let $\Bbbk$ be a field. Given a finitely generated associative $\Bbbk$-algebra $A$, its growth function $\gamma_A(n)$ is defined as follows. Let $V$ be any finite-dimensional subspace of $A$ which generates it (namely, the span of a finite number of generators). Then $\gamma_A(n)=\dim_\Bbbk V^{\leq n}$, the dimension of the space spanned by all products of at most $n$ generators. This function is independent of the choice of $V$ when considered up to asymptotic equivalence. For more on growth of algebras, see \cite{KL}.

\begin{cor}
Let $f\colon \mathbb{N}\rightarrow \mathbb{N}$ be a non-decreasing submultiplicative function such that $\lim_{n\rightarrow \infty} f(n)=\infty$. Then there exists a finitely generated simple algebra $A$ over an arbitrary field, such that $nf(n)\preceq \gamma_A(n) \preceq n^2 f(n)$.
\end{cor}

\begin{proof}
Let $\Bbbk$ be an arbitrary base field.
Consider a minimal subshift $\mathcal{X}$ over an alphabet $X$ as in Proposition \ref{prop:minimal}, whose complexity function satisfies $f(n)\preceq p_{\mathcal{X}}(n) \preceq n f(n)$.

We can find a single bi-infinite word $\omega \in \Sigma^{\mathbb{Z}}$ sharing the same finite factors with $\mathcal{X}$ and thus, replacing $\mathcal{X}$ by the closure of the shift-orbit of $\omega$, we may regard $\mathcal{X}$ as a minimal $\mathbb{Z}$-subshift. 

Nekrashevych \cite{N} associated with every uniformly recurrent bi-infinite word $w$, a finitely generated simple algebra $A_w$ such that $\gamma_{A_w}(n) \sim np_{\mathcal{X}}(n)$. In fact, $A_w$ arises as the convolution algebra $A_w=\Bbbk[\mathfrak{G}_\mathcal{X}]$ of the \'etale groupoid $\mathfrak{G}_\mathcal{X}$ of the action $\mathbb{Z}\curvearrowright \mathcal{X}$, where $\mathcal{X}$ is the closure of the shift-orbit of $w$.

Now consider $w = \omega$. By \cite{N}, $A_\omega$ is a simple $k$-algebra with growth rate $\gamma_A(n)\sim np_\omega(n)=np_{\mathcal{X}}(n)$. It follows that
$$ nf(n)\preceq \underbrace{np_{\mathcal{X}}(n)}_{\sim \gamma_A(n)} \preceq n^2 f(n), $$
as claimed.
\end{proof}

\end{document}